\documentclass[11pt]{amsart}
\usepackage[utf8]{inputenc}
\usepackage[margin=3cm]{geometry}
\usepackage{amsmath,amssymb,amsthm}
\usepackage{enumerate}
\usepackage{graphicx}
\usepackage[english]{babel}
\usepackage[pdfborder={0 0 0}]{hyperref}
\usepackage[utf8]{inputenc}
\usepackage{verbatim}
\usepackage[square,numbers]{natbib}
\usepackage{doi}
\usepackage{color}
\usepackage{chngcntr}
\usepackage{pgfplots}
\pgfplotsset{compat=1.18}
\usepackage{mathrsfs}
\usetikzlibrary{arrows}

\usepackage{cleveref}
\usepackage{hyperref}

\usepackage{setspace}

\newtheorem{theorem}{Theorem}
\newtheorem*{theorem*}{Theorem}
\newtheorem{proposition}[theorem]{Proposition}

\newtheorem{lemma}[theorem]{Lemma}
\newtheorem{conjecture}[theorem]{Conjecture}
\newtheorem*{conjecture*}{Conjecture}

\newtheorem*{question*}{Question}

\theoremstyle{remark}
\newtheorem{remark}[theorem]{Remark}
\theoremstyle{definition}

\newtheorem*{shearer}{Shearer's Theorem}
\newtheorem*{mainthmredux}{Theorem~\ref{thm:main} redux}
\newtheorem*{mainconjredux}{Conjecture~\ref{conj:occfrac} redux}

\newcommand{\icount}{i} 
\newcommand{\ratioexp}{\varphi}

\title{Triangle-free graphs with the fewest independent sets}

\author{Pjotr Buys}
\address{Korteweg--de Vries Institute for Mathematics, University of Amsterdam, Netherlands.\footnote{\,Current address: School of Mathematical Sciences, Queen Mary University of London, UK.}}
\email{\protect\href{mailto:pjotr.buys@gmail.com}{\protect\nolinkurl{pjotr.buys@gmail.com}}}
\author{Jan van den Heuvel}
\address{Department of Mathematics, London School of Economics \& Political Science, London, UK.}
\email{\protect\href{mailto:j.van-den-heuvel@lse.ac.uk}{\protect\nolinkurl{j.van-den-heuvel@lse.ac.uk}}}
\author{Ross J. Kang}
\address{Korteweg--de Vries Institute for Mathematics, University of Amsterdam, Netherlands. }
\email{\protect\href{mailto:r.kang@uva.nl}{\protect\nolinkurl{r.kang@uva.nl}}}

\begin{document}

\begin{abstract}
Given $d>0$ and a positive integer $n$, let $G$ be a triangle-free graph on $n$ vertices with average degree $d$.
With an elegant induction, Shearer (1983) tightened a seminal result of Ajtai, Koml\'os and Szemer\'edi (1980/1981) by proving that $G$ contains an independent set of size at least $(1+o(1))\frac{\log d}{d}n$ as $d\to\infty$.

By a generalisation of Shearer's method, we prove that the number of independent sets in $G$ must be at least $\exp\bigl((1+o(1))\frac{(\log d)^2}{2d}n\bigr)$ as $d\to\infty$.
This improves upon results of Cooper and Mubayi (2014) and Davies, Jenssen, Perkins, and Roberts (2018).
Our method also provides good lower bounds on the independence polynomial of $G$, one of which implies Shearer's result itself.
As certified by a classic probabilistic construction, our bound on the number of independent sets is sharp to several leading terms as $d\to\infty$.
\end{abstract}


\date{\today}
\maketitle

\section{Introduction}\label{sec:intro}

\noindent
Given a graph $G=(V,E)$, an \emph{independent set} is a subset of $V$ that spans no edge of $G$ and a \emph{triangle} is a three-vertex subset that spans three edges of $G$.
For $v\in V$, we write~$N(v)$ for the set of vertices neighbouring $v$ and $\deg(v)=|N(v)|$ for the degree of $v$.
We denote the average degree $\frac{1}{|V|}\sum_{v \in V} \deg(v)$ of $G$ by $d(G)$ (where we treat this as $0$ if $G$ is empty).
We study the collection $\mathcal{I}(G)$ of all independent sets of $G$ (including the empty set).

We address the following basic question:
\begin{quote}
\emph{What triangle-free $G=(V,E)$ minimises $|\mathcal{I}(G)|$ in terms of $|V|$ and $d(G)$?}
\end{quote}
In a meaningful and precise sense, we show that if we draw $G$ as a binomial random graph with expected average degree $d$, then with positive probability we essentially obtain the minimiser.

By Tur\'an's theorem~\cite{Tur41},  the cardinality $\alpha(G)$ of a largest element in $\mathcal{I}(G)$, also known as the \emph{independence number} of $G$, satisfies $\alpha(G)\ge \frac{1}{d(G)+1}|V|$, and this is exact for $G$ some disjoint union of complete graphs.
On the other hand, it is known that forbidding some fixed graph $H$ as a subgraph of $G$ ensures that $G$ contains some significantly larger independent set (see~\cite{AEKS81,She95}).
Due to its implications for the off-diagonal Ramsey numbers (see~\cite{BoKe21,FGM20}) and its central role in the study of the structure of locally sparse graphs (see~\cite{DaKa25+,DKPS20+}), the case $H=K_3$ the triangle is of prime importance.

Seminal work of Ajtai, Koml\'os and Szemer\'edi~\cite{AKS80,AKS81} showed that  $\alpha(G) = \Omega\bigl(\frac{\log d(G)}{d(G)}|V|\bigr)$ for any triangle-free $G$.
Soon after, in an acclaimed work, Shearer~\cite{She83} sharpened this as follows.
Note that in this paper a logarithm is always the natural logarithm (with base ${e}$).

\begin{shearer}[\cite{She83}]\label{thm:shearer}\mbox{}\\*
\textit{If $G=(V,E)$ is a  triangle-free graph of average degree $d$, then $\alpha(G) \ge (1+o(1)) \dfrac{\log d}{d}|V|$ as $d\to \infty$.}
\end{shearer}

\noindent
A standard analysis of the binomial random graph shows this bound to be sharp up to a $2+o(1)$ multiplicative factor (see\ e.g.~\cite{JLR00}; see also Remark~\ref{rem:shearersharpness}).
Despite sustained attention  through the decades --- in part because it incidentally yields the best-known upper bound for the Ramsey numbers $R(3,k)$ --- Shearer's bound remains the state of the art.

Interest has grown in another basic property of $\mathcal{I}(G)$, namely its cardinality $\icount(G)=|\mathcal{I}(G)|$.
Most notably, interesting \emph{upper} bounds on $\icount(G)$ were motivated by questions in combinatorial number theory (see~\cite{Alo91,Cam87,Kah01,Zha10}).
Surprisingly, systematic investigation into lower bounds on $\icount(G)$ has begun in earnest only relatively recently~\cite{CoMu14,CuRa14}, eventually helping to precipitate some remarkable insights into $\alpha(G)$ and more, especially through the partition function and elegant properties of the hard-core model (see~\cite{DJPR17,DJPR18,DKPS20+}).

It was observed in~\cite{CoMu14} (see also~\cite{CuRa14,SSSZ19}) that the number of independent sets of any graph~$G$ must satisfy $\log( \icount(G)) =  \Omega\bigl(\frac{\log d(G)}{d(G)}|V|\bigr)$, and this is sharp up to the leading constant, again with $G$ as some disjoint union of complete graphs.
Cooper and Mubayi~\cite{CoMu14} also showed that, mirroring the situation for $\alpha(G)$, there is a substantially better guarantee on $\icount(G)$ if $G$ is triangle-free; $\log(\icount(G)) =  \Omega\bigl(\frac{(\log d(G))^2}{d(G)}|V|\bigr)$ in that case.
Later, with Dutta~\cite{CDM14}, they proved that Shearer's Theorem itself may be used to show $\log(\icount(G)) \ge (1+o(1))\frac{(\log d(G))^2}{4d(G)}|V|$, where the $o(1)$ term is as $d(G)\to\infty$.
Improving on both this bound and another one due to Davies, Jenssen, Perkins, and Roberts~\cite[Thm.~2]{DJPR18}, our main result is as follows.

\begin{theorem}\label{thm:main}\mbox{}\\*
If $G=(V,E)$ is a  triangle-free graph of average degree $d$, then $\log( \icount(G)) \ge  (1+o(1))  \dfrac{(\log d)^2}{2d}|V|$ as $d\to \infty$.
\end{theorem}

\noindent
Davies et al.~\cite{DJPR18} leveraged local properties of the hard-core model to prove an analogous result involving the \emph{maximum degree $\Delta(G)$} rather than the average degree $d(G)$.
Moreover, they showed their bound on $\log(\icount(G))$ to be asymptotically sharp --- that is, the only slack is in the $o(1)$ term.
The same is thus true of the bound in Theorem~\ref{thm:main}.

In fact, we establish Theorem~\ref{thm:main} not only in a more general sense --- that is, for the independence polynomial of $G$, which we define shortly --- but also in a more precise one.
In Section~\ref{sec:sharpness}, we spell out an analysis of the random graph to show our bound is sharp to several terms in addition to the leading asymptotic term.
Here is the more precise restatement.

Recall that the \emph{Lambert~$W$ function} (also called the \emph{product logarithm}) is the inverse of the function $\mathbb{R}_{\geq 0} \to \mathbb{R}_{\geq 0}$ given by $x \mapsto xe^{x}$.
It is real analytic on $\mathbb{R}_{\geq 0}$, and satisfies $W(x) = \log x - \log\log x + O\bigl((\log\log x)/\log x\bigr) = \log x - \log\log x + o(1)$ as $x\to\infty$.

\begin{mainthmredux}\mbox{}\\*
\textit{Let $\mathcal{F}_1: \mathbb{R}_{\geq 0}\to \mathbb{R}_{\geq 0}$ be the function defined by
\[\mathcal{F}_1(d) = \inf\Bigl\{\frac{\log(\icount(G))}{|V|} : \text{$G=(V,E)$ a triangle-free graph of average degree at most $d$}\Bigr\}.\]
Then for $d\geq 0$ we have
\[\frac{W(d)^2+2W(d)-W(2)^2-2W(2)}{2(d - 2)} \le \mathcal{F}_1(d) \le \frac{W(d)^2+2W(d)}{2d}.\]}
\end{mainthmredux}

\begin{figure}[ht]
    \centering
    \begin{tikzpicture}
        \begin{axis}[
            grid=major,
            legend pos=north east,
            width=15cm,
            height=5cm,
            xmin=0,
            xmax=100,
            ymin=0,
            ymax=1,
            xtick={0,20,40,60,80,100},
            ytick={0,0.25,0.5,0.75,1}
        ]
            \addplot[smooth,
                line width=1pt,
                color = orange!50!magenta] table[x index=0, y index=2] {plot_data.dat};
            \addlegendentry{$\frac{W(d)^2+2W(d)}{2d}$}

            \addplot[smooth,
                line width=1pt,
                color=blue!50!cyan
                ]
                table[x index=0, y index=1] {plot_data.dat};
            \addlegendentry{$\frac{W(d)^2+2W(d)-W(2)^2-2W(2)}{2(d - 2)}$}

        \end{axis}
    \end{tikzpicture}
    \caption{A comparison plot of the boundary curves for $\mathcal{F}_1(d)$ in Theorem~\ref{thm:main} redux.}
    \label{fig:comparisonplot}
\end{figure}
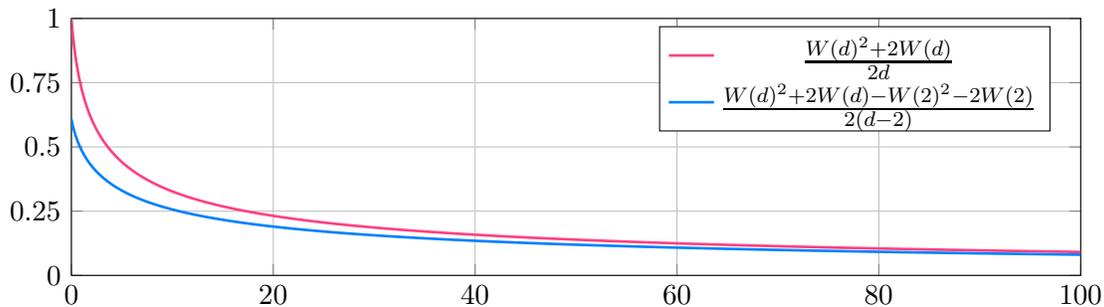

\medskip
\noindent
The difference between these upper and lower bounds on $\mathcal{F}_1(d)$ is $O(1/d)$; see Figure~\ref{fig:comparisonplot}.
Such sharpness is rare in the literature\footnote{\,A notable example is the work related to the Lower Matching Conjecture, solved by Csikv\'ari~\cite{Csi17}, about the fewest matchings of a given size in $d$-regular bipartite graphs.
Although it is in a much more structured setting, this matches the random bipartite $d$-regular multigraphs to an even finer precision than what we see here.}.
To the best of our knowledge, there are no other Ramsey-type results in combinatorics matching a probabilistic construction to this precision.
It is unclear to us whether so precise a result could be found as corollary to an occupancy fraction bound such as \cite[Thm.~1]{DJPR18}, which applies to triangle-free graphs of maximum degree $\Delta$.
Integration of that result (see \cite[Eq.~(3)]{DJPR18}) gives a bound shy of sharpness already in the second-order term as $\Delta\to\infty$.

Our proof of Theorem~\ref{thm:main} is inspired by Shearer's original proof in~\cite{She83}, and we feel it shares some of the same grace.
Indeed, our proof also implies Shearer's Theorem (see Remark~\ref{rem:shearer}).
Shearer performed a strikingly simple induction on the number of vertices, finessed by the probabilistic method.
By observing how the average degree changes when removing the closed neighbourhood of a uniformly chosen random vertex, Shearer's proof eventually reduced the bound to the differential inequality
\begin{equation}\label{eq:shearer}
1+(x-x^2)f'(x)\ge (x+1)f(x).
\end{equation}
That is, he showed that, if $f$ is a non-increasing differentiable convex function that satisfies~\eqref{eq:shearer}, then $\alpha(G) \ge f(d)|V|$ for any triangle-free graph $G=(V,E)$ with average degree $d$~\cite[Thm.~1]{She83}.
It is thankfully rather simple to find such a function that satisfies~\eqref{eq:shearer} with equality.

However, we will shortly see in our case that, as we are dealing with $i(G)$ instead of $\alpha(G)$, the analogous approach yields something more complicated.
Although finding a good function here was certainly no foregone conclusion, requiring an inspired choice, we surmise there will be further interesting insights through the development of Shearer's ingenious method.

\subsection*{The hard-core model}

\mbox{}\\
Given a graph $G=(V,E)$ and $\lambda\ge0$, consider a random independent set $X$ of~$G$, where any $I\in \mathcal{I}(G)$ is chosen with probability proportional to $\lambda^{|I|}$, that is,
\[\Pr(X = I) = \frac{\lambda^{|I|}}{\sum_{J\in\mathcal{I}(G)} \lambda^{|J|}}.\]
This distribution is called the \emph{hard-core model at fugacity $\lambda$}, and the normalising factor
\[Z_G(\lambda) = \sum_{J\in\mathcal{I}(G)} \lambda^{|J|}\]
is called its \emph{partition function}.
This function is also known as the \emph{independence polynomial} of~$G$.
Note that the uniform case corresponds to $\lambda=1$ and that $\icount(G) = Z_G(1)$.

Writing the \emph{closed neighbourhood} $\{v\} \cup N(v)$ of $v$ as $N[v]$,
it is crucial to observe that, for any vertex $v\in V$, we have  $Z_G(\lambda) = Z_{G-v}(\lambda) + \lambda Z_{G-N[v]}(\lambda)$, which  for $\lambda=1$ corresponds to $\icount(G) = \icount(G - v) + \icount(G - N[v])$.

Our primary objective is to get a good handle on the following extremal parameter:
\[\mathcal{F}_\lambda(d) = \inf\Bigl\{\frac{\log(Z_G(\lambda))}{|V|} : \text{$G=(V,E)$ a triangle-free graph of average degree at most $d$}\Bigr\}.\]
As discussed above, the best-known upper bounds on the Ramsey numbers $R(3,k)$ are based on our understanding of another extremal parameter:
\[\widetilde{\mathcal{F}}_\infty(d) = \inf\Bigl\{\frac{\alpha(G)}{|V|} : \text{$G=(V,E)$ a triangle-free graph of average degree at most $d$}\Bigr\}.\]
Shearer's Theorem says that $\widetilde{\mathcal{F}}_\infty(d) \ge (1+o(1))  \frac{\log d}{d}$ as $d\to\infty$, or more generally that $\widetilde{\mathcal{F}}_\infty(d) \ge f(d)$ for any non-increasing differentiable convex function $f$ that satisfies~\eqref{eq:shearer}.

Let us briefly point out how the two extremal parameters are related.

\begin{lemma}\label{lem:largelambda}\mbox{}\\*
\textit{For $d\ge 0$ we have $\lim\limits_{\lambda\to\infty} \dfrac{\mathcal{F}_\lambda(d)}{\log(1+\lambda)}=\widetilde{\mathcal{F}}_\infty(d)$.
}\end{lemma}

\begin{proof}
For $\varepsilon >0$, let $G=(V,E)$ be a triangle-free graph of average degree at most $d$ so that $\alpha(G)/|V| \le \widetilde{\mathcal{F}}_\infty(d)+\varepsilon$.
A polynomial $p$ with a positive leading coefficient satisfies $\lim\limits_{\lambda \to \infty} \tfrac{\log (p(\lambda))}{\log(1+\lambda)} = \deg(p)$, and so
\[\lim_{\lambda\to\infty}\frac{\mathcal{F}_\lambda(d)}{\log(1+\lambda)} \le \lim_{\lambda\to\infty}\frac{\log (Z_G(\lambda))}{|V|\log(1+\lambda)} = \frac{\alpha(G)}{|V|} \le \widetilde{\mathcal{F}}_\infty(d)+\varepsilon.\]
As $\varepsilon$ was chosen arbitrarily, this proves one inequality.
The opposite inequality follows easily from the fact that $Z_G(\lambda)\ge(1+\lambda)^{\alpha(G)}$, as every subset of a maximum independent set is independent, and so $\log(Z_G(\lambda)) \ge \alpha(G)\log(1+\lambda)$ for any $G$.
\end{proof}

\section{Proof}\label{sec:proof}

\noindent
We prove Theorem~\ref{thm:main} in the more general context of the partition function $Z_G(\lambda)$ for fixed $\lambda\in[0,2]$, as embodied by Theorem~\ref{thm:mainlambda} below.
The proof consists of an induction \emph{\`a la} Shearer (Lemma~\ref{lem:mainlambda}), followed by an analytic `exercise' (Lemma~\ref{lem:holygrail}).
The latter step uses that $\lambda\in[0,2]$.

\begin{lemma}\label{lem:mainlambda}\mbox{}\\*
    Let $\lambda \geq 0$, and let $f_\lambda : \mathbb{R}_{\geq 0} \to \mathbb{R}$ be a non-increasing differentiable convex function such that for all $x \in \mathbb{R}_{\geq 0}$ we have
    \begin{equation}
        \label{eq:lambdaineq}
        \exp(-x f_\lambda'(x) - f_\lambda(x)) + \lambda \cdot \exp((x-x^2)f_\lambda'(x) - (x+1)f_\lambda(x)) \geq 1.
    \end{equation}
    Then if $G=(V,E)$ is a  triangle-free graph of average degree $d$, $\log( Z_{G}(\lambda)) \ge f_\lambda(d)\cdot|V|$.
\end{lemma}

\noindent
Stated differently, Lemma~\ref{lem:mainlambda} asserts that $\mathcal{F}_\lambda(d) \ge f_\lambda(d)$ for any non-increasing differentiable convex function $f_\lambda$ that satisfies~\eqref{eq:lambdaineq}.

The reader is invited to ponder the differential inequality of Lemma~\ref{lem:mainlambda} and compare it to that of~\eqref{eq:shearer} from Shearer's proof.
Even for $\lambda=1$, solving~\eqref{eq:lambdaineq} with equality seems rather tricky.

The induction below is directly analogous to Shearer's, except that the parameter of consideration is number of independent sets, rather than size of a largest independent set.

\begin{proof}[Proof of Lemma~\ref{lem:mainlambda}]
Take $n=|V|$.
For $n = 0$ the statement reduces to $\log( Z_{K_0}(\lambda))\ge 0$, i.e.\ $Z_{K_0}(\lambda) \geq 1$, where $K_0$ denotes the empty graph, which is indeed correct.

Now we assume that $n \geq 1$ and that the statement holds for graphs with strictly less than~$n$ vertices.
Let $G=(V,E)$ be a graph on $n$ vertices and average degree $d$.
Let $v$ be an arbitrary vertex of $G$.
We use the following shorthand notation: $G_v^{\mathrm{out}}$ denotes $G-v$ and~$G_v^{\mathrm{in}}$ denotes $G-N[v]$.
And we use $n_v^{\mathrm{out}}, n_v^{\mathrm{in}}$, $m_v^{\mathrm{out}}, m_v^{\mathrm{in}}$, $d_v^{\mathrm{out}}$, and $d_v^{\mathrm{in}}$ for the number of vertices, number of edges, and the average degrees of the graphs $G_v^{\mathrm{out}}$ and $G_v^{\mathrm{in}}$, respectively.
So we have $n_v^{\sigma} d_v^\sigma = 2 m_v^{\sigma}$ for $\sigma \in \{\mathrm{out},\mathrm{in}\}$.

Convexity of $f_\lambda$ gives $f_\lambda(d_v^{\sigma}) \geq f_\lambda(d) + (d_v^\sigma - d)f_\lambda'(d)$ for $\sigma \in \{\mathrm{out},\mathrm{in}\}$, and by induction
\begin{align*}
    Z_G(\lambda) &= Z_{G_v^{\mathrm{out}}}(\lambda) + \lambda \cdot Z_{G_v^{\mathrm{in}}}(\lambda)
    \geq e^{f_\lambda(d_v^{\mathrm{out}})n_v^{\mathrm{out}}} + \lambda e^{f_\lambda(d_v^{\mathrm{in}})n_v^{\mathrm{in}}} \\
    &\geq e^{f_\lambda(d)n} \Bigl[e^{(n_v^{\mathrm{out}}-n)f_\lambda(d)+n_v^{\mathrm{out}}(d_v^{\mathrm{out}}-d) f_\lambda'(d)} + \lambda e^{(n_v^{\mathrm{in}}-n)f_\lambda(d)+n_v^{\mathrm{in}}(d_v^{\mathrm{in}}-d) f_\lambda'(d)} \Bigr] \\
    & = e^{f_\lambda(d)n} \Bigl[e^{(n-n_v^{\mathrm{out}})(d f_\lambda'(d) - f_\lambda(d)) - 2(m-m_v^{\mathrm{out}}) f_\lambda'(d)} + \lambda e^{(n-n_v^{\mathrm{in}})(d f_\lambda'(d) - f_\lambda(d)) - 2(m-m_v^{\mathrm{in}}) f_\lambda'(d)} \Bigr].
\end{align*}
The proof is completed if we can show that there exists a vertex $v$ such that the term between square brackets is at least $1$.

Let $\mathbf{v}$ be a vertex chosen uniformly at random.
We show that the required inequality holds in expectation.
Observe that $\mathbb{E}[(n - n_{\mathbf{v}}^{\mathrm{out}})] = 1$, $\mathbb{E}[(n - n_{\mathbf{v}}^{\mathrm{in}})] = d + 1$, $\mathbb{E}[(m - m_{\mathbf{v}}^{\mathrm{out}})] = d$, and
\[\mathbb{E}[(m - m_{\mathbf{v}}^{\mathrm{in}})] = \sum_{uw \in E} \hspace{-3pt}\mathbb{P}[\mathbf{v} \in N(u) \cup N(w)] = \sum_{uw \in E} \frac{\deg(u) + \deg(w)}{n}    = \frac{1}{n}\sum_{u \in V} \deg(u)^2 \geq d^2,\]
where the last inequality follows from the Cauchy--Schwarz inequality, and the second equality is the only place in the proof where the condition of triangle-freeness is needed.
We thus have
\begin{align*}
    \mathbb{E}\bigl[(n-n_{\mathbf{v}}^{\mathrm{out}})(d f_\lambda'(d) - f_\lambda(d)) - 2(m-m_{\mathbf{v}}^{\mathrm{out}}) f_\lambda'(d)\bigr] & = (d f_\lambda'(d) - f_\lambda(d)) - 2df_\lambda'(d)\\
    &= -df_\lambda'(d) - f_\lambda(d); \\
    \mathbb{E}\bigr[(n-n_{\mathbf{v}}^{\mathrm{in}})(d f_\lambda'(d) - f_\lambda(d)) - 2(m-m_{\mathbf{v}}^{\mathrm{in}}) f_\lambda'(d)\bigr] &\geq (d+1)(d f_\lambda'(d) - f_\lambda(d)) - 2d^2f_\lambda'(d)\\
    &= (d-d^2)f_\lambda'(d) - (d+1) f_\lambda(d);
\end{align*}
where we have used $f_\lambda'(d) \leq 0$.
The proof is now concluded by using linearity of expectation, Jensen's inequality applied to the exponential, and inequality~\eqref{eq:lambdaineq}.
\end{proof}

\noindent
The main difficulty in the rest of the proof is to find a non-increasing convex function $f_\lambda$ that satisfies both the differential inequality~\eqref{eq:lambdaineq} and $f_\lambda(x) = (1+o(1))\tfrac{(\log x)^2}{2x}$ as $x\to\infty$.
Even only in the case $\lambda=1$, this task is far from obvious.
Our exposition below omits a graveyard lined with attempted choices for $f_\lambda$, many that could only be reasonably verified by eyeing plots output by mathematical software.

The choice of $f_\lambda$ we obtained eventually came about as follows.
Inequality~\eqref{eq:lambdaineq} can be written in the factored form
\[e^{-xf_\lambda'(x)-f_\lambda(x)} \cdot \bigl[1 + \lambda\cdot e^{-x f_\lambda(x) - x(x-2)f_\lambda'(x)}\bigr] \geq 1.\]
We tried setting the exponent in the bracket equal to several simple functions of $x$ and solving explicitly for $f_\lambda$.
The choice $-x f_\lambda(x) - x(x-2)f_\lambda'(x)=-W(\lambda x)$ turned out to work, yielding the function in the lemma below.

\begin{lemma}\label{lem:holygrail}\mbox{}\\*
    For $\lambda \in [0,2]$, writing  $c_\lambda = W(2\lambda) + \tfrac12 W(2\lambda)^2$, the following function satisfies the conditions in Lemma~\ref{lem:mainlambda}:
    \[f_\lambda (x) = \frac{W(\lambda x) + \tfrac12 W(\lambda x)^2 - c_\lambda}{x - 2}.\]
\end{lemma}

\begin{proof}
The choice of $c_\lambda$ is made to remove the singularity at $x=2$, so that $f_\lambda$ is analytic on~$\mathbb{R}_{\geq 0}$.
We will prove that $f_\lambda$ is non-increasing, convex, and satisfies inequality \eqref{eq:lambdaineq}.
For $\lambda = 0$ these clearly hold, as we have $f_{0}$ is identically $0$.
We may thus assume that $\lambda > 0$.

Write
	\[P(x) = e^{-W(x)} = \frac{W(x)}{x},\quad\text{and let $P_{\lambda}(x) = \lambda \cdot P(\lambda x)$.}\]
Note that $P$ is convex and decreasing, as it is the exponential applied to the convex and decreasing function $-W$.
So $P_\lambda$ too is convex and decreasing.

Next we have that $\int P_\lambda(y)\,dy= W(\lambda y)+\frac12 W(\lambda y)^2+ C$, using the substitution $u=W(\lambda y)$.
	So we can rewrite $f_\lambda$ to
\[f_\lambda(x)= \frac{1}{x-2} \int_{2}^x P_\lambda(y)\, dy.\]
It follows that $f_{\lambda}$ is convex, since it is the running average of a convex function.

Taking derivatives of the expression above gives
\begin{equation}
	\label{eq:3}f_{\lambda}'(x) = \frac{P_\lambda(x) - f_{\lambda}(x)}{x - 2}.
\end{equation}
Because $P_\lambda$ is decreasing, we find that for $x \geq 2$ the value $P_{\lambda}(x)$ is always at most its running average $f_{\lambda}$, and thus $f'_{\lambda}(x)$ is nonpositive.
For $x < 2$ the same comparison holds with both signs flipped.
We conclude that $f_\lambda$ is non-increasing for $x\ge0$.

Recall that inequality~\eqref{eq:lambdaineq} can be written in factored form as
\[e^{-xf_\lambda'(x)-f_\lambda(x)} \cdot \bigl[1 + \lambda\cdot e^{-x f_\lambda(x) - x(x-2)f_\lambda'(x)}\bigr] \geq 1.\]
Using expression~\eqref{eq:3} we can write
\[-x f_\lambda'(x) - f_\lambda(x)=\frac{2f_\lambda(x)-xP_\lambda(x)}{x-2}\]
and
\[-xf_\lambda(x)-x(x-2)f_\lambda'(x)= -xP_\lambda(x)= -W(\lambda x).\]
So~\eqref{eq:lambdaineq} can be simplified to
\[\exp\Bigl(\frac{2f_{\lambda}(x) - x P_{\lambda}(x)}{x-2}\Bigr) \cdot (1 + P_{\lambda}(x)) \geq 1.\]
We subsequently take the logarithm on both sides and we use $\log(1+y) \geq 2y / (y+2)$ for non-negative $y$ to find that it is sufficient to prove that
\[\frac{2f_{\lambda}(x) - x P_{\lambda}(x)}{x-2} + \frac{2P_{\lambda}(x)}{2+P_\lambda(x)} \geq 0.\]
We write the left-hand side as $\tfrac{t_{\lambda}(x)}{(x-2)^2}$.
The derivative of $t_\lambda$ is
\[t_{\lambda}'(x) = \Bigl(\frac{P_\lambda(x)}{2+P_{\lambda}(x)}\Bigr)^2 \cdot \frac{r_{\lambda}(x)}{1+xP_{\lambda}(x)} \cdot (x-2),\]
where $r_{\lambda}(x) = 4 - 2 P_\lambda(x) - x P_{\lambda}(x)^2$.
For this we use that the Lambert W function has derivative $W'(x)=\dfrac{-W(x)}{x(1+W(x))}$, and hence $P_\lambda$ has derivative
\[P_\lambda'(x)= \frac{-W(\lambda x)^2}{x^2(1+W(\lambda x))}= \frac{-P_\lambda(x)^2}{1+xP_\lambda(x)}.\]
We observe that $r_{\lambda}'(x) = P_\lambda(x)^2\geq 0$ (using the derivative of $P_\lambda$ again).
We also have $r_{\lambda}(0) = 4 - 2\lambda \geq 0$, as $\lambda \leq 2$, and thus $r_{\lambda}(x)$ is non-negative for $x \geq 0$.

It follows that $t_{\lambda}$ is non-increasing for $x \leq 2$ and non-decreasing for $x \geq 2$ and thus it has its global minimum at $x = 2$, where it is equal to zero.
This concludes the proof.
\end{proof}

\noindent
The two lemmas are the two halves of the following result.

\begin{theorem}\label{thm:mainlambda}\mbox{}\\*
Let $\lambda\in[0,2]$.
If $G=(V,E)$ is a  triangle-free graph of average degree $d$, then
\begin{equation}
    \label{eq:lowbound}
    \log( Z_{G}(\lambda)) \ge \frac{W(\lambda d)^2+2W(\lambda d)-W(2\lambda)^2-2W(2\lambda)}{2(d - 2)}\cdot|V|.
\end{equation}
\end{theorem}

\medskip
\noindent
Stated differently, Theorem~\ref{thm:mainlambda} asserts that, if $\lambda\in[0,2]$, then
\[\mathcal{F}_\lambda(d) \ge \frac{W(\lambda d)^2+2W(\lambda d)-W(2\lambda)^2-2W(2\lambda)}{2(d - 2)}.\]
This bound specialised to $\lambda=1$ implies Theorem~\ref{thm:main} as in its first statement.
In fact, the asymptotic bound on $\icount(G)=Z_G(1)$ is also valid for $Z_G(\lambda)$ for any fixed $\lambda\in(0,2]$.

\begin{remark}
The condition $\lambda \le 2$ is not special.
Through a more sophisticated analysis, we can prove that \autoref{lem:holygrail} (and thus also \autoref{thm:mainlambda}) extends to $\lambda \in [0, \lambda_M]$, where $\lambda_M \approx 2.61$.
Furthermore, numerical evidence suggests that \autoref{lem:holygrail} remains valid for $\lambda$ up to approximately $\lambda_M \approx 11.65$.
This would be close to a theoretical limit for \autoref{thm:mainlambda} because of the following.
For the edgeless graph $G=(V,\varnothing)$ we have $\log(Z_{G}(\lambda)) = \log(1+\lambda) \cdot |V|$, while the right-hand side of inequality~\eqref{eq:lowbound} for $d = 0$ simplifies to $\frac{1}{4}(W(2\lambda)^2+2 W(2\lambda)) \cdot |V|$.
For $\lambda \geq 13.971$ the latter expression is greater than the former.
\end{remark}

\begin{remark}
Despite the previous remark, Lemma~\ref{lem:largelambda} suggests we should inspect the bound~\eqref{eq:lowbound} for large $\lambda$, hypothetically speaking.
We have
\[\frac{W(\lambda d)^2+2W(\lambda d)-W(2\lambda)^2-2W(2\lambda)}{2(d - 2)} = \Bigl(\frac{\log d-\log 2}{d-2}+o(1)\Bigr)\log\lambda\quad\text{as $\lambda\to\infty$}.\]
This would correspond to a  lower bound on $\widetilde{\mathcal{F}}_\infty(d)$ that is too large for smaller values of~$d$.
(Certainly we must exclude $d$ up to $4$, for example, due to work in~\cite{LPRR11}.)
But we cannot rule it out as a possibly valid bound for larger $d$ yet.
By way of comparison, the lower bound on the independence polynomial (albeit in the maximum degree setting) in~\cite[Thm.~4]{DJPR18} corresponds to an analogous expression of order $(\log\log \lambda)^2$ as $\lambda\to\infty$.
\end{remark}

\begin{remark}\label{rem:shearer}
If $f$ is any non-increasing differentiable convex function satisfying~\eqref{eq:shearer}, then the function $d \mapsto f(d) \log(1+\lambda)$ satisfies the hypothesis of Lemma~\ref{lem:mainlambda} for any $\lambda \geq 0$.
By Lemma~\ref{lem:mainlambda} we have $\mathcal{F}_\lambda(d) \ge f(d)\log(1+\lambda)$, and thus by Lemma~\ref{lem:largelambda} we have $\widetilde{\mathcal{F}}_\infty(d) = \lim\limits_{\lambda\to\infty}\tfrac{\mathcal{F}_\lambda(d)}{\log(1+\lambda)} \ge f(d)$.
In this way, Lemma~\ref{lem:mainlambda} is a strengthening of Shearer's Theorem in its sharpest form.

This raises the following analytic question, which could test the strength of our induction as compared to Shearer's original induction. For $d\ge 0$ and $\lambda\ge 0$, define
\begin{align*}
\tilde{g}(d)
& =  \sup\{f(d) : \text{$f$ non-increasing differentiable convex satisfying~\eqref{eq:shearer}}\};\\
g_\lambda(d)
& = \sup\{f_\lambda(d) : \text{$f_\lambda$ non-increasing differentiable convex satisfying~\eqref{eq:lambdaineq}}\}.
\end{align*}
It is not difficult to show that $\tilde{g}$ is in fact the choice of $f$ that Shearer~\cite{She83} found to solve~\eqref{eq:shearer} with equality.
Is there some $d>0$ for which
$\limsup\limits_{\lambda\to\infty} \tfrac{g_\lambda(d)}{\log(1+\lambda)} > \tilde{g}(d)$?
\end{remark}

\section{Sharpness}\label{sec:sharpness}

\noindent
The following probabilistic construction backstops the precise restatement of Theorem~\ref{thm:main}, as well as an analogous version of Theorem~\ref{thm:mainlambda}, by giving upper bounds on $\mathcal{F}_\lambda(d)$.

\begin{proposition}\label{prop:construction}\mbox{}\\*
For $\lambda\ge 0$, $d>0$ define
\[\ratioexp(\lambda,d)=
    \begin{cases}
\frac{W(\lambda d)^2+2W(\lambda d)}{2d} & \text{if  $\log\lambda \le d$}; \\
1 - \frac12 d + \log\lambda & \text{otherwise}.
\end{cases}\]
Then for all $\varepsilon>0$ we have that for sufficiently large $n$, there exists a triangle-free graph $G=(V,E)$ with $|V|=n$ and $d(G) < d+\varepsilon$ such that $\log(Z_{G}(\lambda)) < \left(\ratioexp(\lambda,d)+\varepsilon\right)|V|$.
\end{proposition}

\begin{proof}
Since $Z_0(G)=1$ for all $G$, the conclusion is trivial if $\lambda=0$.
We now assume $\lambda>0$.

Let $\mathbf{G}\sim G(n,d/n)$ be distributed as the random graph on $n$ vertices where each edge is included independently at random with probability $d/n$.
We will perform various routine random graph computations; see~\cite{JLR00} for background.
It is a standard fact that
\begin{equation}\label{eq:avgdegprob}
\Pr(d(\mathbf{G}) \ge d+\varepsilon) \to 0\quad \text{as $n\to\infty$}.
\end{equation}
Writing $X$ for the number of triangles in $\mathbf{G}$, it is known (see~\cite[Sec.~3.3]{JLR00}) that
\begin{equation}\label{eq:triangleprob}
\Pr(X=0) \to \exp(-d^3/6)\quad \text{as $n\to\infty$}.
\end{equation}

The expectation of the partition function of the hard-core model on $\mathbf{G}$ at fugacity $\lambda$ satisfies
\begin{align*}
\mathbb{E}(Z_{\mathbf{G}}(\lambda))
& = \sum_{k=0}^n \binom{n}{k} \Bigl(1- \frac{d}{n}\Bigr)^{\binom{k}{2}}\lambda^k
\le  \sum_{k=0}^n \Bigl(\frac{en}{k}\Bigr)^k e^{-\frac{d}{n}\binom{k}{2}}\lambda^k \\
&\le  \sqrt{e^d} \sum_{k=0}^n \exp\Bigl[k\Bigl(1+\log n+\log \frac{\lambda}{k}-\frac{dk}{2n}\Bigr)\Bigr].
\end{align*}
Let us now optimise the exponent in the summand, by maximising it according to the scaling $k=\eta n$, $\eta\in[0,1]$.
Thus we write the exponent as
\begin{equation}\label{eq:exponent}
\Bigl(\eta+\eta\log\frac{\lambda}{\eta}-\tfrac12 d\eta^2\Bigr)n.
\end{equation}
By computing the derivative with respect to~$\eta$ and observing concavity for $\eta\in[0,1]$, we find that the maximum for~\eqref{eq:exponent} occurs when
\[d=\frac{1}{\eta}\log\frac{\lambda}{\eta}\text{ and }\eta\in[0,1],
\quad \text{which has solution } \eta= \begin{cases}
\frac{W(\lambda d)}{d}, & \text{if $\log\lambda \le d$}; \\
1, & \text{otherwise}.
\end{cases}\]
This means that~\eqref{eq:exponent} has the maximum value
\[(\eta+\tfrac12 d\eta^2)n= \Bigl(\frac{2W(\lambda d)+W(\lambda d)^2}{2d}\Bigr)n\]
if $\log\lambda \le d$, and the maximum value
\[(1 - \tfrac12 d + \log\lambda)n\]
otherwise.
Substituting this maximisation back into the summation, we obtain that
\[\mathbb{E}(Z_{\mathbf{G}}(\lambda))\le \sqrt{e^d}\, n \exp[(\ratioexp(\lambda,d))n] \le \exp[(\ratioexp(\lambda,d)+\tfrac12\varepsilon)n],\]
provided $n$ is taken sufficiently large.
By Markov's inequality,
\begin{equation}\label{eq:partitionprob}
\Pr\bigl(Z_{\mathbf{G}}(\lambda) \ge \exp[(\ratioexp(\lambda,d)+\varepsilon)n]\bigr) \le \exp(-\tfrac12\varepsilon n) \to 0\quad \text{as $n\to\infty$}.
\end{equation}

By the probabilistic method, the bounds~\eqref{eq:avgdegprob},~\eqref{eq:triangleprob}, and~\eqref{eq:partitionprob} give the result.
\end{proof}

\begin{remark}\label{rem:shearersharpness}
The computations above also yield sharpness up to a $2+o(1)$ factor for Shearer's Theorem.
This follows by noting that for $\lambda=1$ and $d\ge 2$ the expression~\eqref{eq:exponent} has a root at
\[\frac12 d=   \frac{1}{\eta}\log\frac{e}{\eta}, \quad \text{which gives $\eta = \dfrac{2W(e d/2)}{d} = (1+o(1))\frac{2\log d}{d}$\quad as $d\to\infty$}.\]
Writing this last expression as $\eta_d$, the expected number of independent sets of size
$(\eta_d +  \varepsilon)n$
in~$\mathbf{G}$ tends to $0$ as $n\to\infty$.
So, by Markov's inequality,
$\Pr\bigl(\alpha(\mathbf{G}) \ge (\eta_d +  \varepsilon)n\bigr) \to 0$ as $n\to\infty$.
\end{remark}

\begin{remark}\label{rem:shearersupremacy}
The maximisation in the proof of Proposition~\ref{prop:construction} shows that the normalised size of an independent set (as a proportion of $n$) that contributes the most to $\icount(\mathbf{G})$ is around
\[ \frac{W(d)}{d} = \frac{\log d-\log\log d+o(1)}{d}\quad \text{as $d\to\infty$}.\]
However, the normalised independent set size promised by Shearer's Theorem in its sharpest form~\cite[Thm.~1]{She83} is
\[f(d) = \frac{d\log d - d+1}{(d-1)^2}=\frac{\log d-1+o(1)}{d}\quad \text{as $d\to\infty$},\]
where $f$ has been chosen to satisfy~\eqref{eq:shearer} with equality.
This comparison indicates that, for all large enough $d$, any lower bound on the average independent set size can \emph{never} match  Shearer's bound on the independence number.
\end{remark}

\section{Conclusion}\label{sec:conclusion}

\noindent
Shearer's Theorem has long represented a benchmark in our understanding of the independence number of triangle-free graphs and the Ramsey numbers $R(3,k)$.
We strengthened Shearer's original proof method to help us understand the number of independent sets in triangle-free graphs, and proved a lower bound that is matched to extraordinary precision by a classic probabilistic construction.

Although in a sense the proof method we used is old, our work succeeds recent advances using the local occupancy method (see~\cite{DJPR17,DJPR18,DKPS20+,DaKa25+}).
Those results matched Shearer's result in the first-order asymptotics via the occupancy fraction in the hard-core model.
Even while we were able to circumvent this route, via a different bound on the independence polynomial that implies Shearer's Theorem in its sharpest form (see Remark~\ref{rem:shearer}), we propose the following conjecture about occupancy fraction.
This would strengthen Theorem~\ref{thm:main} essentially, by integration, and~\cite[Thm.~1]{DJPR18}, and it is related to~\cite[Conj.~A]{DaKa25+}.

\begin{conjecture}\label{conj:occfrac}\mbox{}\\*
If $G=(V,E)$ is a  triangle-free graph of average degree $d$, then
\[\frac{Z_G'(1) }{Z_G(1)} \ge (1+o(1))\frac{\log d}{d} |V|\quad \text{as $d\to\infty$}.\]
\end{conjecture}

\medskip
\noindent
Considering the derivative of the right-hand side in inequality~\eqref{eq:lowbound} (see~\cite[Eq.~(1)]{DJPR18}), we suspect the following more general and precise restatement of Conjecture~\ref{conj:occfrac} could be true.

\begin{mainconjredux}\mbox{}\\*
\textit{Let $\lambda\in(0,1]$.
If $G=(V,E)$ is a  triangle-free graph of average degree $d$, then
\[\frac{\lambda Z_G'(\lambda) }{Z_G(\lambda)} \ge \frac{W(\lambda d)-W(2\lambda)}{d-2} |V|.\]}
\end{mainconjredux}

\medskip
\noindent
By the probabilistic method, Conjecture~\ref{conj:occfrac} implies Shearer's Theorem in the form stated in the introduction, though not in its sharpest form (see Remark~\ref{rem:shearersupremacy}).
If one is unconcerned with the asymptotic leading constant, let alone further terms, note that Theorem~\ref{thm:main} already implies an $\Omega\bigl(\frac{\log d}{d} |V|\bigr)$ bound for $Z_G'(1)/Z_G(1)$ (see~\cite[Lem.~19]{DKPS20+}).

\medskip
The proof of Theorem~\ref{thm:mainlambda} --- more specifically, the induction in Lemma~\ref{lem:mainlambda} --- hinges on the relation $Z_G(\lambda) = Z_{G-v}(\lambda) + \lambda Z_{G-N[v]}(\lambda)$, and as such is a rather local method.
One might wonder if this method could be generalised, with some analytic vim, just as the local occupancy method has been adapted to many important classes of locally sparse graphs (see~\cite{DaKa25+,DKPS20+,DJKP21}).

\subsection*{Acknowledgements}\mbox{}\\*
Pjotr Buys was supported by grant OCENW.M20.009 of the Dutch Research Council (NWO). Ross Kang was partially supported by grant OCENW.M20.009 of NWO and the Gravitation Programme NETWORKS (024.002.003) of the Dutch Ministry of Education, Culture and Science (OCW).

The research for this paper was done while Jan van den Heuvel was a visitor at the Korteweg--de Vries Institute. Jan would like to thank the institute, and in particular the members of the Discrete Mathematics \& Quantum Information group for their hospitality and the friendly and inspiring atmosphere; and for the cookies, of course.

We are grateful to Ewan Davies, Matthew Jenssen, Will Perkins, and Guus Regts for helpful comments on an earlier version of this manuscript.
We also thank the two anonymous referees for their careful reading of the paper.

\subsection*{Open access statement}\mbox{}\\*
For the purpose of open access, a CC BY public copyright license is applied to any Author Accepted Manuscript (AAM) arising from this submission.

\bibliographystyle{abbrv}
\bibliography{supershearer}

\begin{thebibliography}{24}
\providecommand{\natexlab}[1]{#1}
\providecommand{\url}[1]{\texttt{#1}}
\expandafter\ifx\csname urlstyle\endcsname\relax
  \providecommand{\doi}[1]{doi: #1}\else
  \providecommand{\doi}{doi: \begingroup \urlstyle{rm}\Url}\fi

\bibitem[Ajtai et~al.(1980)Ajtai, Koml\'{o}s, and Szemer\'{e}di]{AKS80}
M.~Ajtai, J.~Koml\'{o}s, and E.~Szemer\'{e}di.
\newblock A note on {R}amsey numbers.
\newblock \emph{J. Combin. Theory Ser. A}, 29\penalty0 (3):\penalty0 354--360,
  1980.
\newblock \doi{10.1016/0097-3165(80)90030-8}.

\bibitem[Ajtai et~al.(1981{\natexlab{a}})Ajtai, Erd{\H o}s, Koml{\'o}s, and
  Szemer{\'e}di]{AEKS81}
M.~Ajtai, P.~Erd{\H o}s, J.~Koml{\'o}s, and E.~Szemer{\'e}di.
\newblock On {{Tur{\'a}n}}'s theorem for sparse graphs.
\newblock \emph{Combinatorica}, 1:\penalty0 313--317, 1981{\natexlab{a}}.
\newblock \doi{10.1007/BF02579451}.

\bibitem[Ajtai et~al.(1981{\natexlab{b}})Ajtai, Koml\'{o}s, and
  Szemer\'{e}di]{AKS81}
M.~Ajtai, J.~Koml\'{o}s, and E.~Szemer\'{e}di.
\newblock A dense infinite {S}idon sequence.
\newblock \emph{European J. Combin.}, 2\penalty0 (1):\penalty0 1--11,
  1981{\natexlab{b}}.
\newblock \doi{10.1016/S0195-6698(81)80014-5}.

\bibitem[Alon(1991)]{Alo91}
N.~Alon.
\newblock Independent sets in regular graphs and sum-free subsets of finite
  groups.
\newblock \emph{Israel J. Math.}, 73:\penalty0 247--256, 1991.
\newblock \doi{10.1007/BF02772952}.

\bibitem[Bohman and Keevash(2021)]{BoKe21}
T.~Bohman and P.~Keevash.
\newblock Dynamic concentration of the triangle-free process.
\newblock \emph{Random Structures Algorithms}, 58:\penalty0 221--293, 2021.
\newblock \doi{10.1002/rsa.20973}.

\bibitem[Cameron(1987)]{Cam87}
P.~J. Cameron.
\newblock Portrait of a typical sum-free set.
\newblock In \emph{Surveys in {C}ombinatorics 1987}, volume 123 of \emph{London
  {Math}. {Soc}. {Lecture} {Note} {Ser}.}, pages 13--42. Cambridge Univ. Press,
  1987.

\bibitem[Cooper and Mubayi(2014)]{CoMu14}
J.~Cooper and D.~Mubayi.
\newblock Counting independent sets in triangle-free graphs.
\newblock \emph{Proc. Amer. Math. Soc.}, 142\penalty0 (10):\penalty0
  3325--3334, 2014.
\newblock \doi{10.1090/S0002-9939-2014-12068-5}.

\bibitem[Cooper et~al.(2014)Cooper, Dutta, and Mubayi]{CDM14}
J.~Cooper, K.~Dutta, and D.~Mubayi.
\newblock Counting independent sets in hypergraphs.
\newblock \emph{Combin. Probab. Comput.}, 23\penalty0 (4):\penalty0 539--550,
  2014.
\newblock \doi{10.1017/S0963548314000182}.

\bibitem[Csikv\'{a}ri(2017)]{Csi17}
P.~Csikv\'{a}ri.
\newblock Lower matching conjecture, and a new proof of {S}chrijver's and
  {G}urvits's theorems.
\newblock \emph{J. Eur. Math. Soc. (JEMS)}, 19\penalty0 (6):\penalty0
  1811--1844, 2017.
\newblock \doi{10.4171/JEMS/706}.

\bibitem[Cutler and Radcliffe(2014)]{CuRa14}
J.~Cutler and A.~J. Radcliffe.
\newblock The maximum number of complete subgraphs in a graph with given
  maximum degree.
\newblock \emph{J. Combin. Theory Ser. B}, 104:\penalty0 60--71, 2014.
\newblock \doi{10.1016/j.jctb.2013.10.003}.

\bibitem[Davies et~al.(2017)Davies, Jenssen, Perkins, and Roberts]{DJPR17}
E.~Davies, M.~Jenssen, W.~Perkins, and B.~Roberts.
\newblock Independent sets, matchings, and occupancy fractions.
\newblock \emph{J. Lond. Math. Soc. (2)}, 96\penalty0 (1):\penalty0 47--66,
  2017.
\newblock \doi{10.1112/jlms.12056}.

\bibitem[Davies et~al.(2018)Davies, Jenssen, Perkins, and Roberts]{DJPR18}
E.~Davies, M.~Jenssen, W.~Perkins, and B.~Roberts.
\newblock On the average size of independent sets in triangle-free graphs.
\newblock \emph{Proc. Amer. Math. Soc.}, 146\penalty0 (1):\penalty0 111--124,
  2018.
\newblock \doi{10.1090/proc/13728}.

\bibitem[Davies et~al.(2021)Davies, de~Joannis~de Verclos, Kang, and
  Pirot]{DJKP21}
E.~Davies, R.~de~Joannis~de Verclos, R.~J. Kang, and F.~Pirot.
\newblock Occupancy fraction, fractional colouring, and triangle fraction.
\newblock \emph{J. Graph Theory}, 97:\penalty0 557--568, 2021.
\newblock \doi{10.1002/jgt.22671}.

\bibitem[Davies and Kang(2025)]{DaKa25+}
E.~Davies and R.~J. Kang.
\newblock The hard-core model in graph theory.
\newblock \emph{arXiv:2501.03379 [math]}, 2025.
\newblock \doi{10.48550/arXiv.2501.03379}.

\bibitem[Davies et~al.(2020)Davies, Kang, Pirot, and Sereni]{DKPS20+}
E.~Davies, R.~J. Kang, F.~Pirot, and J.-S. Sereni.
\newblock Graph structure via local occupancy.
\newblock \emph{arXiv:2003.14361 [math]}, 2020.
\newblock \doi{10.48550/arXiv.2003.14361}.

\bibitem[Fiz~Pontiveros et~al.(2020)Fiz~Pontiveros, Griffiths, and
  Morris]{FGM20}
G.~Fiz~Pontiveros, S.~Griffiths, and R.~Morris.
\newblock The triangle-free process and the {{Ramsey}} number {$R(3,k)$}.
\newblock \emph{Mem. Amer. Math. Soc.}, 263:\penalty0 1--125, 2020.
\newblock \doi{10.1090/memo/1274}.

\bibitem[Janson et~al.(2000)Janson, {\L}uczak, and Rucinski]{JLR00}
S.~Janson, T.~{\L}uczak, and A.~Rucinski.
\newblock \emph{Random {G}raphs}.
\newblock Wiley-Intersci. Ser. Discrete Math. Optim. Wiley-Intersc., New York,
  2000.
\newblock \doi{10.1002/9781118032718}.

\bibitem[Kahn(2001)]{Kah01}
J.~Kahn.
\newblock An entropy approach to the hard-core model on bipartite graphs.
\newblock \emph{Combin. Probab. Comput.}, 10:\penalty0 219--237, 2001.
\newblock \doi{10.1017/S0963548301004631}.

\bibitem[L{\"o}wenstein et~al.(2011)L{\"o}wenstein, Pedersen, Rautenbach, and
  Regen]{LPRR11}
C.~L{\"o}wenstein, A.~S. Pedersen, D.~Rautenbach, and F.~Regen.
\newblock Independence, odd girth, and average degree.
\newblock \emph{J. Graph Theory}, 67:\penalty0 96--111, 2011.
\newblock \doi{10.1002/jgt.20518}.

\bibitem[Sah et~al.(2019)Sah, Sawhney, Stoner, and Zhao]{SSSZ19}
A.~Sah, M.~Sawhney, D.~Stoner, and Y.~Zhao.
\newblock The number of independent sets in an irregular graph.
\newblock \emph{J. Combin. Theory Ser. B}, 138:\penalty0 172--195, 2019.
\newblock \doi{10.1016/j.jctb.2019.01.007}.

\bibitem[Shearer(1983)]{She83}
J.~B. Shearer.
\newblock A note on the independence number of triangle-free graphs.
\newblock \emph{Discrete Math.}, 46\penalty0 (1):\penalty0 83--87, 1983.
\newblock \doi{10.1016/0012-365X(83)90273-X}.

\bibitem[Shearer(1995)]{She95}
J.~B. Shearer.
\newblock On the independence number of sparse graphs.
\newblock \emph{Random Structures Algorithms}, 7:\penalty0 269--271, 1995.
\newblock \doi{10.1002/rsa.3240070305}.

\bibitem[Tur\'{a}n(1941)]{Tur41}
P.~Tur\'{a}n.
\newblock Eine {E}xtremalaufgabe aus der {G}raphentheorie.
\newblock \emph{Mat. Fiz. Lapok}, 48:\penalty0 436--452, 1941.

\bibitem[Zhao(2010)]{Zha10}
Y.~Zhao.
\newblock The number of independent sets in a regular graph.
\newblock \emph{Combin. Probab. Comput.}, 19:\penalty0 315--320, 2010.
\newblock \doi{10.1017/S0963548309990538}.

\end{thebibliography}

\end{document}